\documentclass[12pt]{article}
\usepackage{amsmath}
\usepackage{amssymb}
\usepackage{amscd}
\usepackage{amsthm}

\theoremstyle{plain}
\newtheorem{Thm}{Theorem}

\newtheorem{Prop}[Thm]{Proposition}
\newtheorem{Cor}[Thm]{Corollary}
\newtheorem{Lem}[Thm]{Lemma}

\theoremstyle{definition}

\theoremstyle{Remark}

\numberwithin{equation}{section}

\title{Hodge theory on generalized normal crossing varieties}

\author{{\em Dedicated to Slava Shokurov for his sixtieth birthday} \\
Yujiro Kawamata}

\begin{document}

\maketitle

\begin{abstract}
We generalize some results in Hodge theory 
to generalized normal crossing varieties.
\end{abstract}

\section{Introduction}

A normal crossing variety is defined to be a variety which is 
locally isomorphic to a normal crossing divisor in a smooth 
variety.
Similarly we define a {\em generalized} normal crossing variety as a 
variety which is locally isomorphic to a direct product of 
normal crossing divisors in smooth varieties.

A suitably defined generalized normal crossing {\em divisor} 
in a generalized normal crossing variety is again a 
generalized normal crossing variety.
In this way a certain induction argument on the dimension works for 
generalized normal crossing
varieties and divisors.
For example a reducible divisor often appears and is transformed
to a normal crossing divisor by resolution of singularities.
Moreover one sometimes has to consider a reducible divisor on a reducible
variety during the induction argument on the dimension as in \cite{invent}.

A generalized normal crossing variety appears naturally in higher dimensional 
stable reduction.
Indeed a fiber of a semistable family is a normal crossing 
variety, and a fiber of a generalized semistable family of Abramovich and Karu 
\cite{AK} is a generalized normal crossing variety.

If each irreducible component of the intersections of some of the 
irreducible components of a generalized normal crossing variety is smooth, 
then it is called a
generalized {\em simple} normal crossing (GSNC) variety.

The purpose of this paper is to prove that some Hodge theoretic statements 
for smooth varieties or pairs can be generalized to 
generalized simple normal crossing varieties or pairs.
It is similar to the extensions  
in \cite{reducible} to simple normal crossing varieties.
Therefore the proofs are similar in many places but we need additional 
care in the details.
For example we have a cell complex 
instead of a simplicial complex in Proposition~\ref{sign}. 
We shall concentrate on the differences from \cite{reducible} in this paper.

The results of this paper are as follows.
We prove first that there exists a naturally defined cohomological mixted
Hodge $\mathbf{Q}$-complex on a projective GSNC pair, a pair consisting
of a GSNC variety and a GSNC divisor (Theorem~\ref{absolute case}).
By the standard machinery due to Deligne (\cite{Deligne2}), we obtain
degenerations of spectral sequences as corollaries.
The point is that we have explicit descriptions of the graded pieces
(Lemma~\ref{Gr1}) so that concrete calculations are possible.

We generalize the above absolute case to the relative case where 
every closed stratum is assumed to be smooth over the base 
(Theorem~\ref{relative case}).
Then by using the theory of canonical extensions, 
we prove that higher direct images of the structure sheaves are
locally free (Theorem~\ref{canonical extension}). 

As an application, we prove a generalization of Koll\'ar's vanishing theorem
(Theorem~\ref{injective}), 
thereby modify the contents of \cite{invent}~\S 4, and correct
an error in the proof of \cite{invent}~Theorem~4.3.
Fujino \cite{Fujino} found a simpler alternative proof of its main theorem, 
but our construction which shows that the original proof still works
might also be useful sometime.

In \S 1, we give definitions concerning the GSNC varieties and GSNC pairs.
We prove some lemmas on desingularizations and coverings.
In \S 2, we construct a cohomological mixed Hodge $\mathbf{Q}$-complex on
GSNC pairs.
The results in \S 2 are generalized to relative settings in \S 3.
We extend the results to the case where there are degenerate fibers in \S 4
by using the theory of canonical extensions.
We prove a vanishing theorem of Koll\'ar type for GSNC
varieties in \S 5 following the original proof of \cite{Kollar1}.
It is generalized to the $\mathbf{Q}$-divisor version
by using the covering lemma, thereby finishing the modification of  
the argument in \cite{invent}.

We work over the base field $\mathbf{C}$.

%%%%%%%%%%%%%%%%%%%%%%%%%%%%%%%%%%%%%%%%%%%%%%%%%%%%%%%%%%%%%%

\section{GSNC pairs}

A reduced complex analytic space $X$ is said to be a 
{\em generalized simple normal crossing variety (GSNC variety)} 
if the following conditions are satisfied:

\begin{enumerate}

\item (local) 
At each point $x \in X$, there exists a complex analytic neighborhood
which is isomorphic to a direct product to normal crossing varieties, 
i.e. varieties isomorphic to normal crossing divisors on smooth varieties.

\item (global) 
Any irreducible component of the intersection of some of the irreducible 
components of $X$ is smooth.

\end{enumerate}

The first condition can be put as follows: 
there exists a complex analytic neighborhood $X_x$ at each point $x$ which is 
embedded into a smooth variety $V$ with a coordinate system such that 
$X_x$ is a complete intersection of divisors defined by monomials of 
coordinates.
The {\em level} of $X$ at $x$ is the smallest number of such equations.
For example, a fiber of a semistable family (\cite{AK}) 
satisfies the first condition.
$X$ is locally complete intersection, hence Gorenstein and locally 
equidimensional.

A {\em closed stratum} of $X$ is an irreducible component 
of the intersection of some of the irreducible components of $X$.
A closed stratum is smooth by assumption.
In the case where $X$ is connected, let $X^{[n]}$ denote, for an integer $n$, 
the disjoint union of all the closed strata of codimension $n$ in $X$.

The combinatorics of the closed strata is described by the {\em dual graph}.
It is a cell complex where a closed stratum of 
codimension $n$ corresponds to an $n$-cell and the inclusion of closed strata
corresponds to the boundary relation in the opposite direction.
Any cell is a direct product of simplices due to the local structure of 
$X$.

\begin{Prop}\label{sign}
There is a Mayer-Vietoris exact sequence
\[
0 \to \mathbf{Q}_X \to \mathbf{Q}_{X^{[0]}} \to \mathbf{Q}_{X^{[1]}} 
\to \dots \to \mathbf{Q}_{X^{[N]}} \to 0 
\]
where $N = \dim X$ and the arrows are alternating sums of restriction
homomorphisms.
\end{Prop}

\begin{proof}
We assign arbitrarily fixed orientation to each cell of the dual graph.
Then the corresponding chain complex has boundary maps 
as alternating sums.
Namely the sign of a boundary map for each pair of closed strata is 
positive or negative according to whether the
orientations are compatible or not.
The sign convention of the sequence of the constant sheaves is 
defined accordingly.
Since each cell is contractible, the chain complex is locally exact, 
hence the sequence of sheaves is exact.
\end{proof}

A Cartier divisor $D$ on $X$ is said to be {\em permissible}
if it does not contain any stratum.
We denote by $D \vert_Y$ the induced Cartier divisor on a closed stratum $Y$.
A {\em generalized simple normal crossing divisor (GSNC divisor)} 
$B$ is a permissible 
Cartier divisor such that $B \vert_Y$ is a reduced simple normal 
crossing divisor for each closed stratum $Y$.
In this case $B$ is again a GSNC variety of codimension $1$ in $X$.
Namely a GSNC divisor on a GSNC variety of level $r$ 
is a GSNC variety of level $r+1$.
Such an inductive structure may be an advantage.
If $D$ is a permissible $\mathbf{Q}$-Cartier divisor whose support is a 
GSNC divisor, then the round up 
$\ulcorner D \urcorner$ is well defined.
We have $\ulcorner D \urcorner \vert_Y = \ulcorner D \vert_Y \urcorner$ 
for any closed stratum $Y$.

A {\em generalized simple normal crossing pair (GSNC pair)} $(X,B)$
consists of a GSNC variety $X$ and a GSNC divisor $B$ on it.
For each closed stratum $Y$ of $X$, we denote $B_Y = B \vert_Y$.
We denote by $B_{X^{[n]}}$ the union of all the simple normal crossing 
divisors $B_Y$ for irreducible components $Y$ of $X^{[n]}$.
A {\em closed stratum} of the pair $(X,B)$ is an irreducible component 
of the intersection of some of the irreducible components of $X$ and $B$.

A morphism $f: X' \to X$ between two GSNC varieties
is said to be a {\em permissible birational morphism}
if it induces a bijection between the sets of closed strata of $X$ and $X'$, 
and birational morphisms between closed strata.
A smooth subvariety $C$ of $X$ is said to be a {\em permissible center} for
a GSNC pair $(X,B)$ if the 
following conditions are satisfied:

\begin{enumerate}

\item $C$ does not contain any closed stratum, 
and the scheme theoretic intersections $C_Y$ of $C$ and the closed strata $Y$ 
are smooth.

\item At each point $x \in C_Y$, there exists a coordinate system 
in a neighborhood of $x$ such that 
$B_Y$ and $C_Y$ are defined by ideals generated by monomials.

\end{enumerate}

A blowing up $f: X' \to X$ whose center is permissible with respect to 
a GSNC pair $(X,B)$ is called a {\em permissible blowing up}.
It is a permissible birational morphism from another GSNC variety, 
and the union of the strict transform 
of $B$ and the exceptional divisor is again a GSNC divisor.

The following is a corollary of a theorem of Hironaka \cite{Hironaka}.

\begin{Lem}\label{resolution}
Let $(X,B)$ be a GSNC pair.
Consider one of the following situations:

(a) Let $f: X \dashrightarrow S$ be a rational map to another variety 
whose domain of definition has non-empty intersection with
each closed stratum of the pair $(X,B)$.

(b) Let $Z$ be a closed subset which does not 
contain any closed stratum of the pair $(X,B)$.

Then there exists a tower of permissible blowing-ups
\[
g: (X_n,B_n) \to (X_{n-1},B_{n-1}) \to \dots \to (X_0,B_0)=(X,B)
\]
where $B_{i+1}$ is the union of the strict transform of $B_i$ and 
the exceptional divisor, satisfying the following, respectively to the 
situations (a) and (b):

(a) There is a morphism $h: X_n \to S$ such that $h = f \circ g$.

(b) The union $\bar B_n = B_n \cup g^{-1}(Z)$ is a 
GSNC divisor on $X_n$.
\end{Lem}

\begin{proof}
(a) We reduce inductively the indeterminacy locus of the rational map $f$.
For an irreducible component $Y$ of $X$,
the restriction of $f$ to $Y$ is resolved by a tower of permissible 
blowing-ups of the pair $(Y,B_Y)$ by the theorem of Hironaka.
Since the indeterminacy locus of $f$ does not contain 
any closed stratum, the centers of the blowing-ups do not contain 
any closed stratum either.
Moreover this property is preserved in the process.
Therefore the same centers determine a tower of permissible 
blowing-ups of $(X,B)$.
If we do the same process for all the irreducible components of $X$, then 
we obtain the assertion.

(b) is similar to (a).
We resolve $Z$ at each irreducible component $Y$ of $X$,
since $Z$ does not contain any closed stratum of the pair $(Y,B_Y)$.
\end{proof}

We generalize a covering lemma \cite{AB}:

\begin{Lem}\label{covering}
Let $(X,B)$ be a quasi-projective GSNC pair.
Let $D_j$ ($j=1,\dots,r$) be permissible reduced Cartier divisors 
which satisfy the following conditions:

\begin{enumerate}
\item $D_j \subset B$. 

\item $D_j$ and $D_{j'}$ have no common irreducible components if $j \ne j'$.

\item $D_j \cap Y$ are smooth for all $D_j$ and closed strata $Y$ of $X$.
\end{enumerate}

Let $d_j$ be rational numbers, and $D = \sum d_jD_j$.
Then there exists a finite surjective morphism from another GSNC pair 
$\pi: (X',B') \to (X,B)$ such that $B' = \pi^{-1}(B)$ and 
$\pi^*D$ has integral coefficients.
\end{Lem}

\begin{proof}
The proof is similar to \cite{AB}~Theorem~17, where 
the $D_j$ play the role of the irreducible components of $D$.
\end{proof}

%%%%%%%%%%%%%%%%%%%%%%%%%%%%%%%%%%%%%%%%%%%%%%%%%%%%%

\section{Hodge theory on a GSNC pair}

We construct explicitly the mixed Hodge structures of Deligne \cite{Deligne2}
for GSNC pairs.

Let $(X,B)$ be a GSNC pair.
We do not assume that $X$ is compact at the moment.
We define a cohomological mixed Hodge complex on $X$ when $X$ is 
projective.
We can extend the construction in \cite{reducible} to a GSNC pair, 
but we use here the de Rham complex of Du Bois.
The underlying local system is the constant sheaf $\mathbf{Q}$ in our case,
while it is different and more difficult in the former case.

We define a {\em de Rham complex} $\bar{\Omega}_X^{\bullet}(\log B)$
by the following Mayer-Vietoris exact sequence:
\[
\begin{split}
&0 \to \bar{\Omega}_X^{\bullet}(\log B) \to 
\Omega^{\bullet}_{X^{[0]}}(\log B_{X^{[0]}}) \to
\Omega^{\bullet}_{X^{[1]}}(\log B_{X^{[1]}}) 
\to \\
&\dots \to \Omega^{\bullet}_{X^{[N]}}(\log B_{X^{[N]}}) \to 0
\end{split}
\]
where $N = \dim X$ and the arrows are the alternating sums of the restriction 
homomorphisms, where the signs are defined as in Proposition~\ref{sign}.
We have $\bar{\Omega}_X^0 \cong \mathcal{O}_X$, and 
\[
Ri_*\mathbf{C}_{X \setminus B} \cong \bar{\Omega}_X^{\bullet}(\log B)
\]
for the open immersion $i: X \setminus B \to X$.

We define a {\em weight filtration} on the complex 
$\bar{\Omega}^{\bullet}_X(\log B)$ by 
\[
\begin{split}
&0 \to W_q(\bar{\Omega}^{\bullet}_X(\log B))
\to W_q(\Omega^{\bullet}_{X^{[0]}}(\log B_{X^{[0]}})) 
\to W_{q+1}(\Omega^{\bullet}_{X^{[1]}}(\log B_{X^{[1]}})) \to \\
&\dots \to W_{q+N}(\Omega^{\bullet}_{X^{[N]}}(\log B_{X^{[N]}})) \to 0
\end{split}
\]
where the $W$'s from the second terms denote 
the filtration with respect to the order of log poles.
For example, $W_q(\Omega^{\bullet}_{X^{[0]}}(\log B_{X^{[0]}}))$ is the 
subcomplex of $\Omega^{\bullet}_{X^{[0]}}(\log B_{X^{[0]}})$ consisting
of logarithmic forms on $X^{[0]}$ whose log poles along $B_{X^{[0]}}$ have
order at most $q$.

Before we define a weight filtration on the $\mathbf{Q}$-level object, 
we recall the definition of a convolution of a complex of objects
in a triangulated category \cite{GM}.
Let
\[
a_0 \to a_1 \to \dots \to a_{n-1} \to a_n
\]
be a complex of objects.
If there exists a sequence of distinguished triangles
\[
b_{t+1}[-1] \to b_t \to a_t \to b_{t+1} \\
\]
for $0 \le t < n$ with an isomorphism $b_n \to a_n$,
then $b_0$ is said to be a {\em convolution} of the complex.
A convolution may not exist and may not be unique if it exists.

We also need to define a {\em canonical filtration} of a complex.
If $a_{\bullet}$ is a complex in an abelian category, then we define
\[
\tau_{\le q}(a_{\bullet})_n = \begin{cases} a_n \text{ if } n < q \\
\text{Ker}(a_q \to a_{q+1}) \\
0 \text{ if } n > q \end{cases}
\]
so that $H^n(\tau_{\le q}(a_{\bullet})) \cong H^n(a_{\bullet})$ if 
$n \le q$ and $\cong 0$ otherwise.

In the same way as \cite{reducible}, 
we can define a {\em weight filtration} $W_q(Ri_*\mathbf{Q}_{X \setminus B})$
as a convolution of the following complex of objects
\[
\begin{split}
&\tau_{\le q}(R(i_0)_*\mathbf{Q}_{X^{[0]} \setminus B_{X^{[0]}}})
\to \tau_{\le q+1}(R(i_1)_*\mathbf{Q}_{X^{[1]} \setminus B_{X^{[1]}}}) \to \\
&\dots \to 
\tau_{\le q+N}(R(i_N)_*\mathbf{Q}_{X^{[N]} \setminus B_{X^{[N]}}})
\end{split}
\]
where $\tau$ denotes the canonical filtration
and $i_p: X^{[p]} \setminus B_{X^{[p]}} \to X^{[p]}$ are open immersions, 
such that there are isomorphisms
\[
W_q(Ri_*\mathbf{Q}_{X \setminus B}) \otimes \mathbf{C} 
\cong W_q(\bar{\Omega}_X^{\bullet}(\log B))
\]
for all $q$.

We define the {\em Hodge filtration} as a stupid filtration:
\[
F^p(\bar{\Omega}^{\bullet}_X(\log B))
= \bar{\Omega}^{\ge p}_X(\log B).
\]
Then we have the following whose proof can be compared to 
\cite{reducible}~Lemma~3.2:

\begin{Lem}\label{Gr1}
\[
\begin{split}
&\text{Gr}_q^W(Ri_*\mathbf{Q}_{X \setminus B})
\cong \bigoplus_{p \ge 0,dim X^{[p]}-\dim Y=p+q} \mathbf{Q}_Y[-2p-q] \\
&\text{Gr}_q^W(\bar{\Omega}^{\bullet}_X(\log B))
\cong \bigoplus_{p \ge 0,dim X^{[p]}-\dim Y=p+q} \Omega_Y^{\bullet}[-2p-q] \\
&F^r(\text{Gr}_q^W(\bar{\Omega}^{\bullet}_X(\log B)))
\cong \bigoplus_{p \ge 0,dim X^{[p]}-\dim Y=p+q} \Omega_Y^{\ge r-p-q}[-2p-q]
\end{split}
\]
where the $p$ run over all non-negative integers and the
$Y$ run over all the closed strata of the pair $(X^{[p]},B_{X^{[p]}})$ 
of codimension $p+q$.
\end{Lem}

\begin{proof}
We have 
\[
\begin{split}
&\text{Gr}_q^W(Ri_*\mathbf{Q}_{X \setminus B})
\cong \bigoplus_p R^{p+q}(i_p)_*\mathbf{Q}_{X^{[p]} \setminus B_{X^{[p]}}}
[-2p-q] \\
&\cong \bigoplus_{p \ge 0,dim X^{[p]}-\dim Y=p+q} \mathbf{Q}_Y[-2p-q]  \\
&\text{Gr}_q^W(\bar{\Omega}^{\bullet}_X(\log B))
\cong \bigoplus_p \text{Gr}_{p+q}^W(\Omega^{\bullet}_{X^{[p]}}
(\log B_{X^{[p]}}))[-p] \\
&\cong \bigoplus_{p \ge 0,dim X^{[p]}-\dim Y=p+q} \Omega_Y^{\bullet}[-2p-q] \\
&F^r(\text{Gr}_q^W(\bar{\Omega}^{\bullet}_X(\log B)))
\cong \bigoplus_p \text{Gr}_{p+q}^W(F^r(\Omega^{\bullet}_{X^{[p]}}
(\log B_{X^{[p]}})))[-p] \\
&\cong \bigoplus_{p \ge 0,dim X^{[p]}-\dim Y=p+q} \Omega_Y^{\ge r-p-q}[-2p-q]
\end{split}
\]
\end{proof}

As a corollary, we have the following whose proof is similar to 
\cite{reducible}~Theorem 3.3:

\begin{Thm}\label{absolute case}
Let $(X,B)$ be a GSNC pair.
Assume that $X$ is projective.
Then 
\[
((Ri_*\mathbf{Q}_{X \setminus B}, W), 
(\bar{\Omega}^{\bullet}_X(\log B), W, F))
\]
is a cohomological mixed Hodge $\mathbf{Q}$-complex on $X$.
\end{Thm}

\begin{proof}
If $Y$ is a closed stratum of the pair $(X^{[p]},B_{X^{[p]}})$ 
of codimension $p+q$, then
\[
(\mathbf{Q}_Y,\Omega_Y^{\bullet},F(-p-q))
\]
is a cohomological Hodge $\mathbf{Q}$-complex of weight $2(p+q)$, where 
$F(-p-q)^r=F^{r-p-q}$.
Hence
\[
(\mathbf{Q}_Y[-2p-q],\Omega_Y^{\bullet}[-2p-q],F(-p-q)[-2p-q])
\]
is a cohomological Hodge $\mathbf{Q}$-complex of weight $2(p+q)-2p-q=q$.
\end{proof}

\begin{Cor}\label{degeneration1}
The weight spectral sequence
\[
{}_WE_1^{p,q} = H^{p+q}(\text{Gr}_{-p}^W(Ri_*\mathbf{Q}_{X \setminus B}))
\Rightarrow H^{p+q}(\mathbf{Q}_{X \setminus B})
\]
degenerates at $E_2$, and the Hodge spectral sequence
\[
{}_FE_1^{p,q} = H^q(\bar{\Omega}_X^p(\log B))
\Rightarrow H^{p+q}(\bar{\Omega}_X^{\bullet}(\log B))
\]
degenerates at $E_1$.
\end{Cor}

\begin{proof}
This is \cite{Deligne3}~8.1.9.
\end{proof}

When $B=0$, we have more degenerations:

\begin{Cor}\label{degeneration10}
The weight spectral sequence of the Hodge pieces
\[
{}_WE_1^{p,q} = H^{p+q}(\text{Gr}_{-p}^W(\bar{\Omega}_X^r))
\Rightarrow H^{p+q}(\bar{\Omega}_X^r)
\]
degenerates at $E_2$ for all $r$.
\end{Cor}

\begin{proof}
The differentials ${}_Wd_1^{p,q}$ of the weight spectral sequence in
the previous corollary preserve the degree of the differentials.
Therefore the $E_2$-degeneration of the third spectral sequence is 
a consequence of the first two degenerations.
\end{proof}

%%%%%%%%%%%%%%%%%%%%%%%%%%%%%%%%%%%%%%%%%%%%%%%%%%%%%%%%%%%%%%%

\section{Relativization}

We can easily generalize the results in the previous section to the relative 
setting when all the closed strata are smooth over the base.

We consider the following situation:
$(X,B)$ is a pair of a GSNC variety and a GSNC divisor, 
$S$ is a smooth algebraic variety, which is not necessarily complete,
and $f: X \to S$ is a projective surjective morphism.
We assume that, for each closed stratum $Y$ of the pair $(X,B)$, 
the restriction $f \vert_Y: Y \to S$ is smooth.

We define a {\em relative de Rham complex} 
$\bar{\Omega}_{X/S}^{\bullet}(\log B)$ by 
the following exact sequence
\[
\begin{split}
&0 \to \bar{\Omega}_{X/S}^{\bullet}(\log B) \to 
\bar{\Omega}_{X^{[0]}/S}^{\bullet}(\log B_{X^{[0]}}) 
\to \bar{\Omega}_{X^{[1]}/S}^{\bullet}(\log B_{X^{[1]}}) \to \\
&\dots \to \bar{\Omega}_{X^{[N]}/S}^{\bullet}(\log B_{X^{[N]}}) 
\to 0.
\end{split}
\]
In particular we have 
\[
\bar{\Omega}^0_{X/S}(\log B) \cong \mathcal{O}_{X}.
\]

A {\em weight filtration} on the complex 
$\bar{\Omega}^{\bullet}_{X/S}(\log B)$ is defined by using the 
filtration with respect to the order of log poles on the closed strata 
as in the previous section:
\[
\begin{split}
&0 \to W_q(\bar{\Omega}_{X/S}^{\bullet}(\log B)) 
\to W_q(\bar{\Omega}_{X^{[0]}/S}^{\bullet}(\log B_{X^{[0]}})) \\
&\to W_{q+1}(\bar{\Omega}_{X^{[1]}/S}^{\bullet}(\log B_{X^{[1]}})) \to \\
&\dots \to W_{q+N}(\bar{\Omega}_{X^{[N]}/S}^{\bullet}(\log B_{X^{[N]}}) 
\to 0.
\end{split}
\]
We define a {\em Hodge filtration} by
\[
F^p(\bar{\Omega}^{\bullet}_X(\log B))
= \bar{\Omega}^{\ge p}_X(\log B).
\]

\begin{Lem}
There is an isomorphism
\[
Ri_*\mathbf{C}_{X \setminus B} \otimes f^{-1}\mathcal{O}_S \cong 
\bar{\Omega}_{X/S}^{\bullet}(\log B)
\]
such that
\[
W_q(Ri_*\mathbf{C}_{X \setminus B}) \otimes f^{-1}\mathcal{O}_S \cong 
W_q(\bar{\Omega}^{\bullet}_{X/S}(\log B)).
\]
\end{Lem}

We have again:

\begin{Lem}
\[
\begin{split}
&\text{Gr}_q^W(\bar{\Omega}^{\bullet}_{X/S}(\log B))
\cong \bigoplus_{p \ge 0,dim X^{[p]}-\dim Y=p+q} \Omega_{Y/S}^{\bullet}[-2p-q] 
\\
&F^r(\text{Gr}_q^W(\bar{\Omega}^{\bullet}_{X/S}(\log B)))
\cong \bigoplus_{p \ge 0,dim X^{[p]}-\dim Y=p+q} \Omega_{Y/S}^{\ge r-p-q}
[-2p-q]
\end{split}
\]
where the $p$ run over all non-negative integers and the
$Y$ run over all the closed strata of the pair $(X^{[p]},B_{X^{[p]}})$ 
of codimension $p+q$.
\end{Lem}

The following theorem and corollaries are 
similar to \cite{reducible}~Theorem~4.1 and Corollary~4.2:

\begin{Thm}\label{relative case}
\[
((R^n(f \circ i)_*\mathbf{Q}_{X \setminus B}, W(-n)), 
(R^nf_*\bar{\Omega}_{X/S}^{\bullet}(\log B), W(-n), F))
\]
is a variation of rational mixed Hodge structures on $S$.
\end{Thm}

\begin{Cor}\label{degeneration2}
The relative weight spectral sequence
\[
{}_WE_1^{p,q} = R^{p+q}f_*\text{Gr}_{-p}^W(Ri_*\mathbf{Q}_{X \setminus B})
\Rightarrow R^{p+q}(f \circ i)_*\mathbf{Q}_{X \setminus B}
\]
degenerates at $E_2$, and the relative Hodge spectral sequence
\[
{}_FE_1^{p,q} = R^qf_*\bar{\Omega}_{X/S}^p(\log B)
\Rightarrow R^{p+q}f_*\bar{\Omega}_{X/S}^{\bullet}(\log B)
\]
degenerates at $E_1$.
\end{Cor}

If $B=0$, then we have again:

\begin{Cor}\label{degeneration20}
The weight spectral sequence of the Hodge pieces
\[
{}_{W,r}E_1^{p,q} = R^{p+q}f_*\text{Gr}_{-p}^W(\bar{\Omega}_{X/S}^r)
\Rightarrow R^{p+q}f_*(\bar{\Omega}_{X/S}^r)
\]
degenerates at $E_2$ for all $r$.
\end{Cor}

%%%%%%%%%%%%%%%%%%%%%%%%%%%%%%%%%%%%%%%%%%%%%%%%%%%%%%%%%%%%%%%

\section{Canonical extension}

We prove local freeness theorem by using the theory 
of canonical extensions when the degeneration locus is a simple 
normal crossing divisor.

Let $(S, B_S)$ be a pair of a smooth projective variety and a 
simple normal crossing divisor.
We denote $S^o = S \setminus B_S$.
Let $H_{\mathbf{Q}}$ be a local system on $S^o$ which underlies a 
variation of mixed Hodge $\mathbf{Q}$-structures.
Assuming that all the local monodromies around the branches of $B_S$ are
quasi-unipotent, we define
the {\em lower canonical extension} $\tilde{\mathcal{H}}$ 
of $\mathcal{H} = H_{\mathbf{Q}} \otimes \mathcal{O}_{S^o}$ as follows.

We take an arbitrary point $s \in B_S$ at the boundary.
Let $N_i=\log T_i$ be the logarithm of the local monodromies $T_i$ around 
the branches of $B_S$ around $s$,
and let $t_i$ be the local coordinates corresponding to the branches.
Here we select a branch of the logarithmic function so that the 
eigenvalues of $N_i$ belong to the interval $2\pi \sqrt{-1}(-1,0]$.
Then the lower canonical extension $\tilde{\mathcal{H}}$ is 
defined as a locally free sheaf on $S$ 
which is generated by local sections of the form 
$\tilde v= \text{exp}(-\sum_i N_i\log t_i/2\pi \sqrt{-1})(v)$ near $s$, 
where the $v$ are flat sections of $H_{\mathbf{Q}}$.
We note that the monodromy actions on $v$ and the logarithmic functions 
are canceled and the $\tilde v$
are single-valued holomorphic sections of $\mathcal{H}$ outside the 
boundary divisors.

By \cite{Schmid}, the Hodge filtration of $\mathcal{H}$ extends
to a filtration by locally free subsheaves, which we denote again by $F$. 

Let $f: X \to S$ be a surjective morphism between smooth projective 
varieties which is smooth over $S^o$.
We denote $X^o=f^{-1}(S^o)$ and $f^o = f \vert_{X^o}$.
Let $H_{\mathbf{Q}}^q = R^qf^o_*\mathbf{Q}_{X^o}$ for an integer $q$, 
let $\mathcal{H}^q = H_{\mathbf{Q}}^q \otimes \mathcal{O}_{S^o}$, and 
let $\tilde{\mathcal{H}}^q$ be the canonical extension.
Then by Koll\'ar \cite{Kollar2} and Nakayama \cite{Nakayama},
there is an isomorphism 
\[
R^qf_*\mathcal{O}_X \cong \text{Gr}_F^0(\tilde{\mathcal{H}}^q)
\]
which extends the natural isomorphism 
\[
R^qf^o_*\mathcal{O}_{X^o} \cong \text{Gr}_F^0(\mathcal{H}^q).
\]

The following theorem will be used in the next section:

\begin{Thm}\label{canonical extension}
Let $X$ be a projective GSNC variety, 
$(S,B_S)$ a pair of a smooth projective variety and a simple normal 
crossing divisor, and let $f: X \to S$ be a projective surjective morphism.
Assume that, for each closed stratum $Y$ of $X$, 
the restriction $f \vert_Y: Y \to S$
is surjective and smooth over $S^o = S \setminus B_S$.
Denote $X^o=f^{-1}(S^o)$ and $f^o = f \vert_{X^o}$.
For integers $q$, let $H_{\mathbf{Q}}^q = R^qf^o_*\mathbf{Q}_{X^o}$ 
be the local system on $S^o$ 
which underlies a variation of mixed Hodge $\mathbf{Q}$-structures
defined in the preceeding section.
Let $\mathcal{H}^q = H_{\mathbf{Q}}^q \otimes \mathcal{O}_{S^o}$, and 
let $\tilde{\mathcal{H}}^q$ be their canonical extensions.
Then the following hold:

(1) The weight spectral sequence of a Hodge piece
\[
{}_{W,0}E_1^{p,q} = R^{p+q}f_*\text{Gr}_{-p}^W(\mathcal{O}_X)
\Rightarrow R^{p+q}f_*\mathcal{O}_X
\]
degenerates at $E_2$.

(2) There are isomorphisms 
\[
R^qf_*\mathcal{O}_X \cong \text{Gr}_F^0(\tilde{\mathcal{H}}^q)
\]
for all integers $q$ which extend the natural isomorphisms 
\[
R^qf^o_*\mathcal{O}_{X^o} \cong \text{Gr}_F^0(\mathcal{H}^q)
\]
in Corollary~\ref{degeneration2}.

(3) $R^qf_*\mathcal{O}_X$ are locally free sheaves.
\end{Thm}

\begin{proof}
We extend the definition of the weight filtration on 
$\mathcal{O}_{X^o} = \text{Gr}_F^0(\bar{\Omega}_{X^o/S^o}^{\bullet})$ 
to $\mathcal{O}_X$ by an exact sequence
\[
\begin{split}
&0 \to W_q(\mathcal{O}_X) \to W_q(\mathcal{O}_{X^{[0]}}) 
\to W_{q+1}(\mathcal{O}_{X^{[1]}}) \to \\
&\dots \to W_{q+N}(\mathcal{O}_{X^{[N]}}) \to 0
\end{split}
\]
where $W_q(\mathcal{O}_{X^{[p]}})=\mathcal{O}_{X^{[p]}}$ for 
$q \ge 0$, and $=0$ otherwise, for any $p$.
By the canonical extension, we extend the $E_2$-degeneration of the weight 
spectral sequence from $S^o$ to $S$ as in \cite{reducible}~Theorem~5.1.
Then we reduce the assertion to the theorem of Koll\'ar and Nakayama.
\end{proof}

%%%%%%%%%%%%%%%%%%%%%%%%%%%%%%%%%%%%%%%%%%%%%%%%%%%%%%%%%%%%%%

\section{Koll\'ar type vanishing}

We shall generalize the vanishing theorem of Koll\'ar \cite{Kollar1} to 
GSNC varieties:

\begin{Thm}\label{injective}
Let $X$ be a projective GSNC variety,
$S$ a normal projective variety, 
$f: X \to S$ a projective surjective morphism, 
and let $L$ be a permissible Cartier divisor on $X$
such that $\mathcal{O}_X(m_1L)$ is generated by global sections
for a positive integer $m_1$.
Assume that $\mathcal{O}_X(m_2L) \cong f^*\mathcal{O}_Z(L_S)$ 
for a positive integer $m_2$ and an ample Cartier divisor $L_S$ on $S$, 
and for each closed stratum $Y$ of $X$, 
the restricted morphism $f \vert_Y: Y \to S$ is surjective.
Then the following hold:

(1) Let $s \in H^0(X,\mathcal{O}_X(nL))$ be a non-zero section for 
some positive integer $n$ such that the corresponding Cartier divisor $D$ is 
permissible.
Then the natural homomorphisms given by $s$
\[
H^p(X, \mathcal{O}_X(K_X+L)) \to H^p(X, \mathcal{O}_X(K_X+L+D)) 
\]
are injective for all $p$.

(2) $H^q(S, R^pf_*\mathcal{O}_X(K_X+L))=0$ for $p \ge 0$ and $q > 0$.

(3) $R^pf_*\mathcal{O}_X(K_X)$ are torsion free for all $p$.
\end{Thm}

\begin{proof}
We follow closely the proof of \cite{Kollar1}~Theorem~2.1 and 2.2.
We use the same numbering of the steps.

{\em Step 1}.
We may assume, and will assume from now, 
that $\mathcal{O}_X(L)$ is generated by global sections.

We achieve this reduction by constructing a Kummer covering and  
taking the Galois invariant part as in loc. cit.

\vskip 1pc

{\em Step 2}.
We prove the dual statement of (1) in the case where $n=1$ and $D$ 
is generic:
Let $s,s' \in H^0(X,L)$ be general members, 
and let $D,D'$ be the corresponding permissible Cartier divisors.
Then the natural homomorphisms given by $s$
\[
H^p(X, \mathcal{O}_X(-D-D')) \to H^p(X, \mathcal{O}_X(-D')) 
\]
are surjective for all $p$.

We go into details in this step in order to show how to 
generalize the argument in loc. cit. to the GSNC case.
Since $s$ and $s'$ are general, $D,D',D+D'$ and $D \cap D'$ are also 
GSNC varieties.
We consider the following commutative diagram:
\[
\begin{CD}
H^{p-1}(\mathcal{O}_{D+D'}) @>>> 
H^p(\mathcal{O}_X(-D-D')) @>>> 
H^p(\mathcal{O}_X) @>d_p>> H^p(\mathcal{O}_{D+D'}) \\
@V{b'_{p-1}}VV @VVV @V=VV @V{b'_p}VV \\
H^{p-1}(\mathcal{O}_{D'}) @>>> 
H^p(\mathcal{O}_X(-D')) @>>> 
H^p(\mathcal{O}_X) @>{e'_p}>> 
H^p(\mathcal{O}_{D'})
\end{CD}
\] 
In order to prove our assertion, we shall prove that 
(a) $b'_{p-1}$ is surjective, and (b) $\text{Ker}(d_p) = \text{Ker}(e'_p)$.

(a) We consider the following Mayer-Vietoris exact sequence
\[
\begin{CD}
H^{p-1}(\mathbf{C}_{D \cap D'}) @>{\bar a_p}>> H^p(\mathbf{C}_{D+D'})
@>{\bar b_p+\bar b'_p}>> H^p(\mathbf{C}_D) \oplus H^p(\mathbf{C}_{D'}) 
@>{\bar c_p-\bar c'_p}>> H^p(\mathbf{C}_{D \cap D'}) 
\end{CD}
\]
whose $\text{Gr}^0_F$ is 
\[
\begin{CD}
H^{p-1}(\mathcal{O}_{D \cap D'}) @>{a_p}>> H^p(\mathcal{O}_{D+D'})
@>{b_p+b'_p}>> H^p(\mathcal{O}_D) \oplus H^p(\mathcal{O}_{D'})
@>{c_p-c'_p}>> H^p(\mathcal{O}_{D \cap D'}).
\end{CD}
\]
There is a path connecting $D$ and $D'$ in the linear system $\vert L \vert$ 
on $X$ which gives a diffeomorphism of pairs $(D,D \cap D') \to 
(D',D \cap D')$ fixing $D \cap D'$.
Therefore we have an equality
$\text{Im}(\bar c_p)=\text{Im}(\bar c'_p)$, 
which implies that $\bar b'_p$ is surjective.
Hence so is $b'_p$.

(b) It is sufficient to prove that $\text{Im}(d_p) \cap \text{Ker}(b'_p) = 0$.
Using a path connecting $D$ and $D'$, we deduce that 
$\text{Ker}(e_p)=\text{Ker}(e'_p)$ for 
$e_p: H^p(\mathcal{O}_X) \to H^p(\mathcal{O}_D)$.
Thus 
\[
\text{Im}(d_p) \cap \text{Ker}(b'_p) = \text{Im}(d_p) \cap \text{Ker}(b_p).
\]
Therefore it is sufficient to prove that 
$\text{Im}(a_p) \cap \text{Im}(d_p) = 0$.
We shall prove that 
\[
\text{Im}(a_p) \cap \text{Im}(d_p) \cap W_q(H^p(\mathcal{O}_{D+D'}))= 0
\]
by induction on $q$.

Let 
\[
\begin{split}
&a_{p,q}: \text{Gr}^W_q(H^{p-1}(\mathcal{O}_{D \cap D'})) \to 
\text{Gr}^W_q(H^p(\mathcal{O}_{D+D'})) \\
&d_{p,q}: \text{Gr}^W_q(H^p(\mathcal{O}_X)) \to 
\text{Gr}^W_q(H^p(\mathcal{O}_{D+D'})).
\end{split}
\]
be the natural homomorphisms.
Then it is sufficient to prove that 
\[
\text{Im}(a_{p,q}) \cap \text{Im}(d_{p,q}) = 0.
\]
For $A=X,D\cap D'$ or $D+D'$, we have the following spectral sequences
\[
{}_{W,A}E_1^{r,s} = H^{r+s}(\text{Gr}_{-r}^W(\mathcal{O}_A))
=\bigoplus_{\dim A - \dim Y = r} H^s(\mathcal{O}_Y)
\Rightarrow H^{r+s}(\mathcal{O}_A). 
\]
Therefore the boundary homomorphism $d_{p,q}$ is induced from the sum of homomorphisms 
$H^s(\mathcal{O}_Y) \to H^s(\mathcal{O}_{Y'})$
such that $s = q$, $Y \subset X$, $Y' \subset D+D'$, $Y' \subset Y$ and
\[
r = \dim X - \dim Y = \dim (D+D') - \dim Y' = p - q
\]
hence $Y' = Y \cap D$ or $Y' = Y \cap D'$.

On the other hand, $a_{p,q}$ is induced from the sum of homomorphisms 
$H^s(\mathcal{O}_{Y''}) \to H^s(\mathcal{O}_{Y'})$
such that $s = q$, $Y'' \subset D \cap D'$, $Y' \subset D+D'$, $Y'' = Y'$ and
\[
r = \dim (D \cap D') - \dim Y'' + 1 = \dim (D+D') - \dim Y' = p - q.
\]
Therefore there is no closed stratum $Y'$ of $D+D'$ which 
receives non-trivial images from both $Y$ and $Y''$, 
hence we have our assertion.

\vskip 1pc

{\em Step 3}.
(1) in the case where $n=2^d-1$ for a positive integer $d$ and $D$ generic
is an immediate corollary of Step 2.

{\em Step 4}.
Proof of (2) is the same as in loc. cit.

\vskip 1pc

{\em Step 5}.
Proof of (3).

This is a generalization of Step 2.
We use the notation there.
$D'$ is again generic but $D$ is special.
More precisely, $D$ is special along a fiber $f^{-1}(s)$ over a point $s \in S$
and generic otherwise.
Therefore the intersection $D \cap D'$ is still generic, hence a 
GSNC variety.

Let $\mu: \tilde X \to X$ be the blowing-up along the center $D \cap D'$.
$\tilde X$ is again a GSNC variety.
We denote by $\tilde Y$ the closed stratum of $\tilde X$ 
above a closed closed stratum $Y$ of $X$.
Let $g: \tilde X \to \mathbf{P}^1$ be the morphism induced from the 
linear system spanned by $D$ and $D'$. 
Let $U \subset \mathbf{P}^1$ be an open dense subset such that the restricted
morphisms $g \vert_{\tilde Y}: \tilde Y \to \mathbf{P}^1$ are smooth over $U$
for all the closed strata $Y$.

(a) In order to prove that $b'_p$ is surjective, 
we need to prove an inclusion 
$\text{Im}(c'_p) \subset \text{Im}(c_p)$.
For this purpose, we shall prove 
\[
\text{Im}(c'_p) = \text{Im}(c_p \circ e_p).
\]
Since $c_p \circ e_p = c'_p \circ e'_p$, the right hand side is contained in
the left hand side.
The other direction is the essential part.
We note that both $c'_p$ and $c_p \circ e_p$ are parts of homomorphisms 
of mixed Hodge structures.

Let $\tilde X_U = g^{-1}(U)$ and $\tilde Y_U = \tilde Y \cap \tilde X_U$.
Let $H^p_{\text{prim}}(\mathbf{C}_{\tilde Y \cap D'}) \subset 
H^p(\mathbf{C}_{\tilde Y \cap D'})$ be the subspace of the 
primitive cohomologies, and $R^pg_{*,\text{prim}}\mathbf{C}_{\tilde Y_U}
\subset R^pg_*\mathbf{C}_{\tilde Y_U}$ the corresponding local subsystem. 
Then natural homomorphisms 
\[
\begin{split}
&R^pg_*\mathbf{C}_{\tilde X_U} \to H^p(\mathbf{C}_{D \cap D'}) \times U \\
&R^pg_{*,\text{prim}}\mathbf{C}_{\tilde Y_U} \to 
H^p(\mathbf{C}_{Y \cap D \cap D'}) \times U
\end{split}
\] 
underlie respectively morphisms of variations of mixed and pure Hodge 
structures over $U$, where the targets are constant variations.

By the semi-simplicity of the category of variations of pure Hodge structures 
of fixed weight (\cite{Deligne2}), 
we deduce that the local system 
$R^pg_{*,\text{prim}}\mathbf{C}_{\tilde Y_U}$ has a 
subsystem which is isomorphic to a constant local system with fiber
$\text{Im}(\bar c'_{p,Y})$ for 
$\bar c'_{p,Y}: H^p_{\text{prim}}(\mathbf{C}_{Y \cap D'}) \to 
H^p(\mathbf{C}_{Y \cap D \cap D'})$.
Then by the invariant cycle theorem (\cite{Deligne2}), we deduce
\[
\begin{split}
&\text{Im}(\bar c'_{p,Y})
\subset \text{Im}(H^0(U,R^pg_{*,\text{prim}}\mathbf{C}_{\tilde Y_U}) \to 
H^p(\mathbf{C}_{Y \cap D \cap D'})) \\
&\subset \text{Im}(H^p(\mathbf{C}_{\tilde Y}) \to 
H^p(\mathbf{C}_{Y \cap D \cap D'}))
\end{split}
\]
where the second and third homomorphisms are derived from the 
restrictions to the fiber $D'$ of $g$.
Since
\[
H^p(\mathcal{O}_{Y \cap D'}) =  
\text{Gr}^0_F(H^p_{\text{prim}}(\mathbf{C}_{Y \cap D'}))
\]
we conclude
\[
\text{Im}(c'_{p,Y}: H^p(\mathcal{O}_{Y \cap D'}) \to 
H^p(\mathcal{O}_{Y \cap D \cap D'}))
\subset \text{Im}(H^p(\mathcal{O}_{\tilde Y}) \to 
H^p(\mathcal{O}_{Y \cap D \cap D'})).
\]
Since the combinatorics of the closed strata are the same for $\tilde X$ and 
$D \cap D'$, we have
\[
\text{Im}(c'_p: H^p(\mathcal{O}_{D'}) \to 
H^p(\mathcal{O}_{D \cap D'}))
\subset \text{Im}(H^p(\mathcal{O}_{\tilde X}) \to H^p(\mathcal{O}_{D \cap D'})).\]
Since $H^p(\mathcal{O}_{\tilde X}) \cong H^p(\mathcal{O}_X)$, we have 
our assertion.

(b) There is an obvious inclusion $\text{Ker}(d_p) \subset \text{Ker}(e'_p)$.
We know already that 
\[
H^p(\mathcal{O}_X(-D-D')) \to H^p(\mathcal{O}_X(-D))
\]
is surjective, because it is a statement which is independent of $D$.
Thus we have $\text{Ker}(d_p) = \text{Ker}(e_p)$.
Therefore it is sufficient to prove 
\[
\text{rank}(e_p) = \text{rank}(e'_p).
\]

As explained in the next Step 6, 
the sheaves $R^pg_*\mathcal{O}_{\tilde X}$ are locally free
for all $p$.
Therefore we have $h^p(\mathcal{O}_D) = h^p(\mathcal{O}_{D'})$, where 
we denote $h^p=\dim H^p$.
We compare two exact sequences 
\[
\begin{split}
&0 \to \mathcal{O}_X(-D) \to \mathcal{O}_X \to \mathcal{O}_D \to 0 \\
&0 \to \mathcal{O}_X(-D') \to \mathcal{O}_X \to \mathcal{O}_{D'} \to 0. 
\end{split}
\]
Since the corresponding $h^p$'s are equal, we deduce that the ranks of the 
corresponding homomorphisms of the long exact sequences are equal,
and we are done.

\vskip 1pc

{\em Step 6}. 
This is our Theorem~\ref{canonical extension}.

{\em Step 7}. 
Proof of (1) is the same as loc. cit.
\end{proof}

Finally we can easily generalize the vanishing theorem to the case of 
$\mathbf{Q}$-divisors by using the covering lemma (Lemma~\ref{covering}):

\begin{Thm}\label{injective2}
Let $X$ be a projective GSNC variety,
$S$ a normal projective variety, 
$f: X \to S$ a projective surjective morphism, 
and let $L$ be a permissible $\mathbf{Q}$-Cartier divisor on $X$
such that $m_1L$ is a Cartier divisor and 
$\mathcal{O}_X(m_1L)$ is generated by global sections
for a positive integer $m_1$.
Assume that the support of $L$ is a GSNC divisor, 
$\mathcal{O}_X(m_2L) \cong f^*\mathcal{O}_Z(L_S)$ 
for a positive integer $m_2$ and an ample Cartier divisor $L_S$ on $S$, 
and for each closed stratum $Y$ of $X$, 
the restricted morphism $f \vert_Y: Y \to S$ is surjective.
Then the following hold:

(1) Let $s \in H^0(X,\mathcal{O}_X(nL))$ be a non-zero section for 
some positive integer $n$ such that the corresponding Cartier divisor $D$ is 
permissible.
Then the natural homomorphisms given by $s$
\[
H^p(X, \mathcal{O}_X(K_X+\ulcorner L \urcorner)) \to 
H^p(X, \mathcal{O}_X(K_X+\ulcorner L \urcorner+D)) 
\]
are injective for all $p$.

(2) $H^q(S, R^pf_*\mathcal{O}_X(K_X+\ulcorner L \urcorner))=0$ 
for $p \ge 0$ and $q > 0$.
\end{Thm}

\begin{proof}
We take a finite Kummer covering $\pi: X' \to X$ from another GSNC variety 
such that $\pi^*L$ becomes a Cartier divisor.
Let $G$ be the Galois group.
Then we have 
\[
(\pi_*\mathcal{O}_X(K_X+L))^G \cong 
\mathcal{O}_X(K_X+\ulcorner L \urcorner).
\]
Therefore our assertion is reduced to the case where $L$ is a Cartier divisor.
\end{proof}

Graduate School of Mathematical Sciences, University of Tokyo,
Komaba, Meguro, Tokyo, 153-8914, Japan 

kawamata@ms.u-tokyo.ac.jp

\end{document}